\documentclass{article}

\usepackage{url}
\usepackage{array,color,textcomp,rotfloat,amsmath,amsthm,amssymb,bigstrut,mathrsfs}
\usepackage[plain]{algorithm}
\usepackage{algpseudocode}

\newtheorem{theorem}{Theorem}[section]
\newtheorem{definition}[theorem]{Definition}
\newtheorem{proposition}[theorem]{Proposition}
\newtheorem{corollary}[theorem]{Corollary}
\newtheorem{lemma}[theorem]{Lemma}

\newtheorem*{question}{Question}

\newcommand{\birthday}{\mathrm{b}}

\newcommand{\N}{\ensuremath{\mathscr{N}}}
\renewcommand{\P}{\ensuremath{\mathscr{P}}}
\newcommand{\linked}{\bowtie}
\newcommand{\mex}{\mathrm{mex}}
\newcommand{\sh}{\textrm{\raisebox{0.85pt}{\tiny \#}}}
\newcommand{\zsp}{{\color{white}0}}
\newcommand{\bd}{\mathrm{b}}
\newcommand{\defn}[1]{\textbf{\textup{#1}}}
\DeclareMathSymbol{\Star}      {\mathord}  {symbols}{"03}
\DeclareMathSymbol{\Neg}       {\mathord}  {symbols}{"00}

\title{On the Structure of Mis\`ere Impartial Games \\ {\Large \textsc{preprint}}}
\author{Aaron N. Siegel}

\begin{document}

\maketitle

\begin{abstract}
We consider the abstract structure of the monoid $\mathcal{M}$ of mis\`ere impartial game values. Several new results are presented, including a proof that the group of fractions of $\mathcal{M}$ is almost torsion-free; a method of calculating the number of distinct games born by day~$7$; and some new results on the structure of prime games. Also included are proofs of a few older results due to Conway, such as the Cancellation Theorem, that are essential to the analysis but whose proofs are not readily available in the literature.

Much of the work presented here was done jointly with John Conway and Dan Hoey, and I dedicate this paper to their memory.
\end{abstract}

\section{Mis\`ere Impartial Games}

The study of impartial combinatorial games has a long history, beginning with Bouton's 1901 solution to Nim~\cite{Bou01} in both normal and mis\`ere play. While the normal-play theory made steady progress in the decades after Bouton, an understanding of mis\`ere play took far longer to come together. The challenges, as we now know, are due primarily to the intrinsic and often counterintuitive complexity of mis\`ere combinatorics.

Grundy and Smith were the first to appreciate the full scope of the difficulties. With evident frustration they wrote, in a seminal 1956 paper~\cite{GrS56},

\begin{quote}
Various authors have discussed the ``disjunctive compound'' 
of games with the last player winning (Grundy~\cite{Gru39}; 
Guy and Smith~\cite{GS56}; Smith~\cite{Smi}).  We attempt here to 
analyse the disjunctive compound with the last player losing, 
though unfortunately with less complete success \dots.
\end{quote}

They understood the proviso and its role in mis\`ere simplification (cf.~Section~\ref{section:simplestform} below), and they held out hope that new techniques might be discovered that would lead to additional simplification rules. Those hopes were dashed by Conway, who proved in the early 1970s that if a game $G$ cannot be simplified via the Grundy--Smith rule, then in fact $G$ is in simplest form. This makes application of the canonical mis\`ere theory essentially intractable in the general case, and subsequent analyses of mis\`ere games have focused on alternative reductions, such as the quotient construction due to Plambeck~\cite{Pla05} and Plambeck and Siegel~\cite{PS08}.

Nonetheless, the canonical mis\`ere theory---despite its limited practical utility---gives rise to a fascinating structure theory. Define mis\`ere game values in the usual manner~\cite[Chapter V.1]{Sie13}:
\[
G = H \quad\text{if}\quad o^-(G+X) = o^-(H+X) \quad\text{for all mis\`ere impartial games } X,
\]
where $o^-(G)$ denotes the mis\`ere outcome of~$G$. The set of mis\`ere game values forms a commutative monoid~$\mathcal{M}$, and it is an alluring problem to study the structure of this monoid for its own sake. Conway proved in the 1970s that $\mathcal{M}$ is cancellative; hence it embeds in its group of differences~$\mathcal{D}$, and we have a rather curious Abelian group that arises cleanly ``in nature.''

Several results on the structure of $\mathcal{M}$ (and~$\mathcal{D}$) are stated in ONaG without proof. The aim of this paper is twofold: first, to gather together what is known about the structure of~$\mathcal{M}$, including proofs of previously known results, into a single narrative; and second, to extend somewhat the frontier of knowledge in this area.

In the former category, we include in particular a proof of the Cancellation Theorem; a derivation of the exact count of $|\mathcal{M}_6|$ (the number of distinct games born by day~$6$); and a proof that every game $G$ can be partitioned, nonuniquely, into just finitely many \textbf{prime parts} (a definition is given below in Section~\ref{section:primes}). All three results are due to Conway~\cite{Con01}, although the count of $|\mathcal{M}_6|$ was stated inaccurately in ONaG and later corrected by Chris Thompson~\cite{Tho99}.

We also offer a smattering of new results:
\begin{itemize}
\item In Section~\ref{section:concubines}, we show that Conway's \textbf{mate} construction is not invariant under equality. In retrospect this should perhaps not be shocking, but it came as a surprise when it was discovered.
\item In Section~\ref{section:dayn} (with additional details given in Appendix~\ref{appendix:daynrecurrences}), we show how to compute $|\mathcal{M}_7|$, the exact number of games born by day~$7$. (The output of this calculation, though it fits comfortably in computer memory, is too large to include in its entirety in a journal article.)
\item In Section~\ref{section:parts}, we show that $\mathcal{D}$ is ``almost'' torsion-free (in precise terms, it is torsion-free modulo association, as defined in Section~\ref{section:parts}). The proof is not especially difficult, but it is non-obvious; Conway previously gave thought to this question, writing in ONaG: ``Further properties of the additive semigroup of games seem quite hard to establish---if $G + G = H + H$ is $G$ necessarily equal to $H$ or $H + 1$?'' (Theorem~\ref{theorem:torsion} implies an affirmative answer.)
\item In Sections~\ref{section:parts} and~\ref{section:primes}, we extend and elaborate on Conway's theory of prime partitions. As one application of this work, we show that all games born by day $6$ have the Unique Partition Property.
\end{itemize}
The first and last results were joint work with John Conway and Dan Hoey, conducted in Princeton during the 2006--07 academic year. In addition, a computational engine for mis\`ere games, written by Hoey as an extension to \emph{cgsuite}, has proved invaluable in assisting the work in this paper.

\section{Prerequisites}

We briefly review the notation and foundational material for mis\`ere games. Results in this section are stated without proof; a full exposition, including proofs for all results stated here, can be found in~\cite[Sections V.1--V.3]{Sie13}.

Formally, an impartial game is identified with the set of its options, so that $G' \in G$ means ``$G'$ is an option of~$G$''. We write $G \cong H$ to mean that $G$ and $H$ are identical as sets; thus it is possible that $G = H$ but $G \not\cong H$.

It is customary to define
\[\begin{array}{lcl}
0 & = & \{\} \bigstrut[b]\\
\Star & = & \{0\} \bigstrut\\
\Star 2 & = & \{0,\ \Star\} \bigstrut\\
\Star m & = & \{0,\ \Star,\ \Star2,\ \ldots,\ \Star(m-1)\} \bigstrut[t]
\end{array}\]
Since only impartial games are under consideration in this paper, we will follow Conway's convention and drop the asterisks, writing
\[0,\ 1,\ 2,\ \ldots,\ m,\ \ldots\]
in place of
\[0,\ \Star,\ \Star2,\ \ldots,\ \Star m,\ \ldots\]
Thus $2 + 2$ is \emph{not} the integer~$4$; it is the game obtained by playing two copies of $\Star2$ side-by-side.

The following conventions will also be used freely. Options of $G$ may be written by concatenation, rather than using set notation; for example, $632$ is the game $\{\Star6,\Star3,\Star2\}$, \emph{not} the integer six hundred thirty-two.  Subscripts denote sums of games: $4_2$ is $\Star 4 + \Star 2$ and $6_{422}$ is $\Star 6 + \Star 4 + \Star 2 + \Star 2$.  Finally, we write $G_\sh$ (pronounced ``$G$~sharp'') for the singleton $\{G\}$.  Sometimes we will employ a chain of sharps and subscripts, which should be read left-to-right; for example:
\[2_{\sh 4 \sh \sh} = (2_\sh + 4)_{\sh \sh}.\]

We write $\mathcal{M}$ for the commutative monoid of all (finite) mis\`ere impartial game values. This monoid can be stratified according to the usual hierarchy:
\begin{definition}
The \defn{formal birthday} $\tilde{\bd}(G)$ of a game $G$ is defined by:
\[\tilde{\bd}(0) = 0; \qquad \tilde{\bd}(G) = \max\left\{\tilde{\bd}(G') + 1 : G' \in G\right\}.\]
The \defn{birthday} $\bd(G)$ is given by
\[\bd(G) = \min\left\{\bd(H) : H = G\right\}.\]
\end{definition}
It is clear from this definition that $\bd(G)$ depends only on the value of~$G$, so that we may write for $n \geq 0$
\[\mathcal{M}_n = \{G \in \mathcal{M} : \bd(G) \leq n\},\]
the set of game values \defn{born by day $n$}, and
\[\mathcal{M} = \bigcup_n \mathcal{M}_n.\]

For a set $X$ we shall write $|X|$ for the cardinality of~$X$, so that $|\mathcal{M}_n|$ is the number of distinct games born by day~$n$.

\subsection*{Mis\`ere Simplification}

The starting point for the canonical mis\`ere theory is the Grundy--Smith \textbf{simplification rule}.

\begin{definition}
Let $G \cong \{G_1',\ldots,G_k'\}$. Let $H$ be a game whose options include those of~$G$:
\[
H \cong \{G_1',\ldots,G_k',H_1',\ldots,H_l'\}.
\]
We say that $H$ \defn{simplifies to} $G$ provided that:
\begin{enumerate}
\item[(i)] $G \in H_j'$ for each $H_j'$ (i.e., each new option $H_j'$ contains a \textbf{reverting move} back to $G$ itself); and
\item[(ii)] if $G \cong 0$, then $o(H) = \N$.
\end{enumerate}
\end{definition}
Clause (ii) is known as the \defn{proviso}.

\begin{theorem}
If $H$ simplifies to $G$, then $G = H$.
\end{theorem}

A simple but useful application of this theorem:

\begin{theorem}[Mis\`ere Mex Rule]
Let $a_1,a_2,\ldots,a_k \in \mathbb{N}$, and suppose that
\[G \cong \{a_1,a_2,\ldots,a_k\}.\]
Then $G = m$, where $m = \mex\{a_1,a_2,\ldots,a_k\}$, \emph{provided that} at least one $a_i$ is $0$ or~$1$.
\end{theorem}
(Here as always, $\mex(X)$ denotes the \textbf{m}inimal \textbf{ex}cluded value of~$X$: the least $m \geq 0$ with $m \not\in X$.)

\subsection*{The Mate of $G$}

\begin{definition}
The \defn{mate} of $G$, denoted by $G^-$, is defined by
\[G^- = \begin{cases}
1 & \textrm{if $G \cong 0$}; \\
\{(G')^- : G' \in G\} & \textrm{otherwise}.
\end{cases}\]
\end{definition}

\begin{proposition}
$G + G^-$ is a (mis\`ere) \P-position.
\end{proposition}

As an easy corollary:

\begin{proposition}
\label{prop:starneq}
$G \neq G + 1$, for every game $G$.
\end{proposition}

\subsection*{The Simplest Form Theorem}
\label{section:simplestform}

Now suppose $G$ has an option $G'$, which in turn has an option $G'' = G$. We say that $G'$ is \defn{reversible through}~$G''$. The Simplest Form Theorem can be stated like so:
\begin{theorem}[Simplest Form Theorem]
Suppose that neither $G$ nor $H$ has any reversible options, and assume that $G = H$. Then $G \cong H$.
\end{theorem}
If $G$ has no reversible options, then we say that $G$ is in \defn{canonical form} (or \defn{simplest form}). It is a remarkable fact that reversible moves can only arise in the context of Grundy--Smith simplification.
\begin{theorem}
Suppose that every option of $H$ is in canonical form and that some option of $H$ is reversible through~$G$. Then $H$ simplifies to~$G$.
\end{theorem}

As in the partizan theory, a constructive test for equality is an essential ingredient in the proof of the Simplest Form Theorem. The details of this constructive test will be independently important, so we review them here as well.

\begin{definition}
We say $G$ is \emph{linked to $H$ (by $T$)}, and we write $G \linked H$, if
\[o(G+T) = o(H+T) = \P\]
for some $T$.
\end{definition}

\begin{lemma}
\label{lemma:linkediffunlike}
$G$ is linked to $H$ if and only if $G =$ no $H'$ and no $G' = H$.
\end{lemma}

\begin{theorem}
\label{theorem:likeiffnotlinked}
$G = H$ if and only if the following four conditions hold:
\begin{enumerate}
\item[(i)] $G$ is linked to no $H'$.
\item[(ii)] No $G'$ is linked to $H$.
\item[(iii)] If $G \cong 0$, then $H$ is an \N-position.
\item[(iv)] If $H \cong 0$, then $G$ is an \N-position.
\end{enumerate}
\end{theorem}

Clauses (iii) and (iv) are, of course, a restatement of the proviso.

\section{Concubines}

\label{section:concubines}

Conway introduced the \defn{mate} $G^-$ as a stepping stone to the Simplest Form Theorem. The terminology perhaps suggests invariance of form, and one might be tempted to suppose that $G = H$ implies $G^- = H^-$, that $G^{--} = G$, and so forth; but in this section we show that essentially all such assertions are false. These are not especially deep observations, but they do not appear to have been pointed out before.

As an example, let $G = (2_{\sh\sh}1)_\sh$. $G$ is easily seen to be in simplest form. However, $G^- = (2_{\sh\sh}0)_\sh$; and since this is an \N-position, the proviso is satisfied and the unique option $2_{\sh\sh}0$ reverses through~$0$.  Therefore $G^- = 0$.  Likewise, if $H \cong (2_{\sh\sh}0)_\sh$, then we have $H = 0$, but $H^- = (2_{\sh\sh}1)_\sh \neq 1$.

\begin{definition}
Suppose that $G$ and $H$ are in simplest form.  We say that $H$ is a \defn{concubine} of $G$ if $H^- = G$, but $G^- \neq H$.
\end{definition}

\begin{proposition}
Every game has a concubine.
\end{proposition}

\begin{proof}
For $G$ in simplest form, we define a game $c(G)$ recursively by
\[c(G) = \begin{cases}
(2_{\sh\sh}1)_\sh & \textrm{if $G = 0$} \\
\{c(G')\} & \textrm{otherwise}
\end{cases}\]
Since the mate of $(2_{\sh\sh}1)_\sh$ is equal to $0$, it is immediately clear that $c(G)^- = G$.  It remains to be shown that $c(G)$ is in simplest form. Suppose (for contradiction) some $c(G)$ is not, and choose $G$ to be a minimal counterexample. Then $c(G)$ simplifies to some game~$H$.

If $c(G')' = c(G'')$ in all cases, then $G$ is obtained from $G''$ by adding reversible options, contradicting the assumption that $G$ is in canonical form.  (The proviso is clearly satisfied, since it is easily seen that $o(G) = o(c(G))$.)  Otherwise, we must have some $G' = 0$, and $H = c(G')' = 2_{\sh\sh}1$.  But this means $c(G)$ must have $1$ as an option.  Since every option of $c(G)$ has the form $c(G')$, and since each $c(G')$ is canonical (by minimality of $G$), this gives the desired contradiction.
\end{proof}

So, for example, $(2_{\sh\sh}1)_\sh$ itself has a concubine, namely
\[(2_{\sh\sh}(2_{\sh\sh}1)_{\sh\sh})_\sh,\]
and indeed there are arbitrarily long chains $G$, $G^-$, $G^{--}$, $G^{---}$, $\ldots$ of distinct games (where it is understood that at each iteration we pass to the canonical form).

\suppressfloats[t]
\section{Games Born By Day $n$}

\label{section:dayn}

A central goal in the structure theory of (any particular class of) combinatorial games is to count the number of game values born by day~$n$. In the most familiar case---normal-play partizan games---and many others, obtaining exact counts is a tedious combinatorial problem, with the only known techniques requiring an exhaustive enumeration of the values being counted. In the impartial mis\`ere theory, by contrast, the simplification rule is surprisingly rigid, and we can write down a recurrence relation for $|\mathcal{M}_n|$ that depends only on an enumeration of $\mathcal{M}_{n-2}$.

\begin{figure}
\[
\begin{array}{c@{\quad=\quad}r}
|\mathcal{M}_0| & 1 \bigstrut \\
|\mathcal{M}_1| & 2 \bigstrut \\
|\mathcal{M}_2| & 3 \bigstrut \\
|\mathcal{M}_3| & 5 \bigstrut \\
|\mathcal{M}_4| & 22 \bigstrut \\
|\mathcal{M}_5| & 4171780
\end{array}
\]

\[
\begin{array}{c@{\quad=\quad}c}
|\mathcal{M}_6| &
\text{\scriptsize$\begin{array}{l@{~}l@{~}l@{~}l@{~}l@{~}l@{~}l@{~}l}
2^{4171780} &
-~2^{2096640} & -~2^{2095104} & -~2^{2094593} & -~2^{2094080} & -~2^{2091523} & -~2^{2091522} \\ &
-~2^{2088960} & -~2^{2088705} & -~2^{2088448} & -~2^{2088193} & -~2^{2086912} & -~2^{2086657} \\ &
-~2^{2086401} & -~2^{2086145} & -~2^{2085888} & -~2^{2079234} & +~2^{1960962} & +~21
\end{array}$}
\end{array}
\]
\caption{The first few values of $|\mathcal{M}_n|$.}
\label{figure:smallbirthday}
\end{figure}

The values $|\mathcal{M}_0|$ through $|\mathcal{M}_5|$ were first calculated by Grundy and Smith in 1956~\cite{GS56}, although $|\mathcal{M}_5| = 4171780$ was not \emph{proven} correct until Conway, armed with the Simplest Form Theorem, came along 20 years later. These initial values are summarized in Figure~\ref{figure:smallbirthday}. Conway added $|\mathcal{M}_6|$ to the list, although the figure initially given in ONaG contained errors; the corrected count, also shown in Figure~\ref{figure:smallbirthday}, was first reported by Chris Thompson in 1999~\cite{Tho99}. In this section, we give \emph{ad hoc} derivations of each of these numbers, then extend the list one step further to compute the exact value of $|\mathcal{M}_7|$.


\subsection*{Games Born By Day $4$}

The earliest games can be enumerated by inspection. The only games born by day $2$ are $0$, $1$, $2$, and $1_\sh$, but $1_\sh = 0$ by the Mex Rule, giving $|\mathcal{M}_2| = 3$. On day~$3$, we have additionally $3$, $20$, $21$, and $2_\sh$, but $20$ and $21$ are reducible (also by the Mex Rule), so there are just two new games, giving $|\mathcal{M}_3| = 5$.

On day~$4$, consider the $2^5$ subsets of $\mathcal{M}_3$. $2^4$ of them contain only nimbers; the remaining $2^4$ have $2_\sh$ as an option. For the subsets containing only nimbers, the canonical games are
\[0,\ 1,\ 2,\ 3,\ 4,\ 2_\sh,\ 3_\sh, \text{ and } 32.\]
(All other combinations of nimbers contain either $0$ or $1$, and therefore reduce to a nimber by the Mex Rule).

Now suppose $2_\sh \in G$ and $G$ is reducible. $G$~must simplify to~$2$, since $2$ is the only option of~$2_\sh$; therefore $0 \in G$ and $1 \in G$, and all other options of $G$ must contain~$2$. The only possibilities are $2_\sh10$ and $2_\sh310$. So of the $2^4$ subsets containing $2_\sh$, two are reducible, and the other fourteen are not. Therefore
\[
|\mathcal{M}_4| = 8 + 14 = 22.
\]

\suppressfloats[t]
\subsection*{Games Born By Day $5$}

There are $2^{22}$ subsets of $\mathcal{M}_4$. To compute $|\mathcal{M}_5|$, we subtract from this total the number of \emph{reducible} subsets of $\mathcal{M}_4$. Now if $H \subset \mathcal{M}_4$ is reducible, then it must simplify to some other game $G \not\cong H$. If we take $G$ to be in simplest form, then by the Simplest Form Theorem it is uniquely determined. Moreover, there must be at least one $H' \in H$ with $G \in H'$, so that necessarily $G \in \mathcal{M}_3$.

This shows that every reducible $H \subset \mathcal{M}_4$ is obtained from a \emph{unique} $G \in \mathcal{M}_3$ by adding reversible moves, so that $H$ has the form
\[
H \cong G \cup \{H_1',\ldots,H_k'\},
\]
with $G \in$ each $H_i'$.

\begin{figure}
\[\begin{array}{c|c|r}
G & |\mathcal{S}_4^G| & \text{Adjustment} \\
\hline
2_\sh & 14 & -2^{14}+1 \bigstrut[t] \\
3 & 10 & -2^{10}+1 \\
2 & 12 & -2^{12}+1 \\
1 & \zsp9 & -2^9+1 \\
0 & 10 & -2^{10}+1 \\
\text{(Proviso)} & & +2^9-1 \\
\hline
\end{array}\]

\[2^{22} - 2^{14} - 2^{10} - 2^{12} - 2^9 - 2^{10} + 2^9 + 4 = 4171780\]
\vspace{-0.5cm}
\caption{The number of games born by day 5. \label{figure:day5}}
\end{figure}

The calculation is summarized in Figure~\ref{figure:day5}. For each $G \in \mathcal{M}_3$, let
\[\mathcal{S}_4^G = \{H \in \mathcal{M}_4 : G \in H\}.\]
When $G \not\cong 0$, then the added reversible moves $\{H_1',\ldots,H_k'\}$ can be \emph{any} nonempty subset of $\mathcal{S}_4^G$, so exactly $2^{|\mathcal{S}_4^G|} - 1$ games simplify to~$G$.

When $G \cong 0$, then the proviso requires additionally that $o(H) = \N$, so that at least one of $H_1',\ldots,H_k'$ must be a \P-position. So if $H_1',\ldots,H_k'$ are all \N-positions with $0 \in H_i'$, then in fact $H$ is canonical. We must therefore \emph{add back} such subsets into the count. This gives rise to the additional ``proviso'' term in Figure~\ref{figure:day5}; the exponent is the count of \N-positions in $\mathcal{M}_4$ that contain $0$ as an option.

\begin{sidewaysfigure}
\newcommand{\ntwon}[1]{-2^{#1} + 1}
\[\begin{array}{c|c@{~}c@{~}c@{~}c@{~}c@{~}c@{~}c@{~}c@{~}c@{~}c@{~}c|r}
G & \multicolumn{11}{c|}{|\mathcal{S}_5^G|} & \multicolumn{1}{c}{\textrm{Adjustment}} \\
\hline
2_\sh3210 & 2^{21} & -~2^{13} & -~2^{11} & -~2^9 & -~2^8 & -~2^8 &&&&&& \ntwon{2085888} \bigstrut[t] \\
2_\sh321  & 2^{21} & -~2^{13} & -~2^{11} & -~2^9 & -~2^8 &       &&&&&& \ntwon{2086144} \\
2_\sh320  & 2^{21} & -~2^{13} & -~2^{11} & -~2^9 &       & -~2^8 &&&&&& \ntwon{2086144} \\
2_\sh32   & 2^{21} & -~2^{13} & -~2^{11} & -~2^9 &       &       &&&&&& \ntwon{2086400} \\
2_\sh31   & 2^{21} & -~2^{13} &          & -~2^9 & -~2^8 &       &&&&&& \ntwon{2088192} \\
2_\sh30   & 2^{21} & -~2^{13} &          & -~2^9 &       & -~2^8 &&&&&& \ntwon{2088192} \\
2_\sh3    & 2^{21} & -~2^{13} &          & -~2^9 &       &       &&&&&& \ntwon{2088448} \\
2_\sh210  & 2^{21} & -~2^{13} & -~2^{11} &       & -~2^8 & -~2^8 &&&&&& \ntwon{2086400} \\
2_\sh21   & 2^{21} & -~2^{13} & -~2^{11} &       & -~2^8 &       &&&&&& \ntwon{2086656} \\
2_\sh20   & 2^{21} & -~2^{13} & -~2^{11} &       &       & -~2^8 &&&&&& \ntwon{2086656} \\
2_\sh2    & 2^{21} & -~2^{13} & -~2^{11} &       &       &       &&&&&& \ntwon{2086912} \\
2_\sh1    & 2^{21} & -~2^{13} &          &       & -~2^8 &       &&&&&& \ntwon{2088704} \\
2_\sh0    & 2^{21} & -~2^{13} &          &       &       & -~2^8 &&&&&& \ntwon{2088704} \\
2_{\sh\sh}& 2^{21} & -~2^{13} &          &       &       &       &&&&&& \ntwon{2088960} \\
32        & 2^{21} &          & -~2^{11} & -~2^9 &       &       &&&&&& \ntwon{2094592} \\
3_\sh     & 2^{21} &          &          & -~2^9 &       &       &&&&&& \ntwon{2096640} \\
4         & 2^{21} &          & -~2^{11} & -~2^9 & -~2^8 & -~2^8 &&&&&& \ntwon{2094080} \\
3         & 2^{21} &          & -~2^{11} &       & -~2^8 & -~2^8 &&&&&& \ntwon{2094592} \\
2_\sh     & 2^{21} &          & -~2^{11} &       &       &       &&&&&& \ntwon{2095104} \\
2         & 2^{21} &          &          &       & -~2^8 & -~2^8 & -(2^{14}\Neg 1) &         & -(2^{10}\Neg 1) & & & \ntwon{2079234} \\
1         & 2^{21} &          &          &       &       & -~2^9 &         & -(2^{12}\Neg 1) & -(2^{10}\Neg 1) & & & \ntwon{2091522} \\
0         & 2^{21} &          &          &       &       &      &         & -(2^{12}\Neg 1) & -(2^{10}\Neg 1) & -(2^9\Neg 1) & & \ntwon{2091523} \\
\text{(Proviso)} & 2^{21} &&&&&&& -(2^{12}\Neg 1) & -(2^{10}\Neg 1) & & -2^{17} & +2^{1960962}-1 \\
\hline
\end{array}\]

\begin{center}
$\begin{array}{c@{}c@{}l@{~}l@{~}l@{~}l@{~}l@{~}l@{~}l@{~}l}
|\mathcal{M}_6| & ~=~ & 2^{4171780} &
-~2^{2096640} & -~2^{2095104} & -~2^{2094593} & -~2^{2094080} & -~2^{2091523} & -~2^{2091522} \\ &&&
-~2^{2088960} & -~2^{2088705} & -~2^{2088448} & -~2^{2088193} & -~2^{2086912} & -~2^{2086657} \\ &&&
-~2^{2086401} & -~2^{2086145} & -~2^{2085888} & -~2^{2079234} & +~2^{1960962} & +~21
\end{array}$
\end{center}
\caption{The number of games born by day 6. \label{figure:day6}}
\end{sidewaysfigure}

\subsection*{Games Born By Day $6$}

The calculation of $|\mathcal{M}_6|$ proceeds similarly and is shown in Figure~\ref{figure:day6}. Just as before, we subtract from $2^{4171780}$ the number of reducible subsets of $\mathcal{M}_5$. Each reducible subset $H \subset \mathcal{M}_5$ is obtained from a unique $G \in \mathcal{M}_4$ by adding reversible moves, and for $G \not\cong 0$ there are exactly $2^{|\mathcal{S}_5^G|}-1$ such possibilities for~$H$, where
\[\mathcal{S}_5^G = \{H \in \mathcal{M}_5 : G \in H\}.\]
As before, the case $G \cong 0$ entails an additional ``proviso term,'' representing those subsets of $\mathcal{S}_5^0$ that are not reducible to $0$ because they are \P-positions.

The calculation of the critical exponents $|\mathcal{S}_5^G|$ is by recursive application of the same principle. For a given $G \in \mathcal{M}_4$, there are exactly $2^{21}$ subsets of $\mathcal{M}_4$ containing~$G$, so we subtract from $2^{21}$ the number of \emph{reducible} such subsets.

Figure~\ref{figure:day6} breaks down the calculation of $|\mathcal{S}_5^G|$ for each~$G$. For $H \subset \mathcal{M}_4$ with $G \in H$, there are two distinct ways that $H$ might be reducible:
\begin{enumerate}
\item[(i)] $H$~simplifies to some $G' \in G$, so that $G$ itself is reversible (together with various other options of~$H$); or
\item[(ii)] $H$~simplifies to some other $K \in \mathcal{M}_3$, with $G \in K$ (so that $G$ is \emph{not} reversible, but remains as an option of the simplified game~$K$).
\end{enumerate}
Case (i) yields one term for each $G' \in G$, and since $G' \in \mathcal{M}_3$, there are at most five such terms (the maximum is achieved for $G = 2_\sh3210$). Case (ii) requires that $G \in \mathcal{M}_2$; hence it is only relevant for $G = 0,1,2$, explaining the special structure of those three rows in Figure~\ref{figure:day6}.

The precise details of how the terms in Figure~\ref{figure:day6} are calculated are fairly subtle. In order not to distract from the flow of this paper, and since the calculation of $|\mathcal{M}_6|$ is only slightly easier than the general case, those details are deferred until Appendix~\ref{appendix:daynrecurrences}.

\suppressfloats[t]
\subsection*{Games Born By Day $7$}

\newcommand\bigpow[1]{
\text{\large 2}^{\left(\text{\scriptsize$
\begin{array}{@{~}l@{~}l@{~}l@{~}l@{~}l@{~}l@{~}l@{~}l@{~}l@{~}}
#1
\end{array}
$}
\right)
}}

\begin{figure}
\[
\begin{array}{c@{~}l}
& \bigpow{
2^{4171780} &
-~2^{2096640} & -~2^{2095104} & -~2^{2094593} & -~2^{2094080} & -~2^{2091523} & -~2^{2091522} \\ &
-~2^{2088960} & -~2^{2088705} & -~2^{2088448} & -~2^{2088193} & -~2^{2086912} & -~2^{2086657} \\ &
-~2^{2086401} & -~2^{2086145} & -~2^{2085888} & -~2^{2079234} & +~2^{1960962} & +~21
} \vspace{0.1in} \\
- & \bigpow{2^{4171779} & -~2^{2085887}} \vspace{0.1in} \\
- & \bigpow{2^{4171779} & -~2^{2086143} & +~1} \vspace{0.1in} \\
- & \bigpow{2^{4171779} & -~2^{2086143} & -~2^{2085887} +~1} \vspace{0.15in} \\
& \qquad \cdots \quad \text{\normalsize ($758656$ additional terms)} \quad \cdots \vspace{0.2in} \\
- & \bigpow{
2^{4171779} &
-~2^{2096640} & -~2^{2094592} & -~2^{2094080} & -~2^{2091522} & -~2^{2091521} & -~2^{2088448} \\ &
-~2^{2088193} & -~2^{2086400} & -~2^{2086145} & -~2^{2085888} & -~2^{2079233} & +~2^{1960961} & +~10
} \vspace{0.1in} \\
+ & \bigpow{
2^{4171779} &
-~2^{3926530} & -~2^{2094592} & -~2^{2094080} & -~2^{2088704} & -~2^{2088192} \\ &
-~2^{2086656} & -~2^{2086400} & -~2^{2086144} & -~2^{2085888} & -~2^{2079234} & +~9
} \vspace{0.1in} \\
+ & 4171779
\end{array}
\]
\caption{Partial expansion of the expression for $|\mathcal{M}_7|$.}
\label{figure:day7}
\end{figure}

The preceding arguments can be abstracted out into a set of recurrence relations for $|\mathcal{M}_n|$ that work for all~$n$. These relations (and a proof that they work) are given below in Appendix~\ref{appendix:daynrecurrences}. They show in particular that $|\mathcal{M}_7|$ has the form
\[
2^{|\mathcal{M}_6|} - 2^{a_1} - \cdots - 2^{a_k} + 2^b + 4171779,
\]
in which all the exponents $a_1,\ldots,a_k,b$ are close to $2^{4171779}$, with~$b$ (the ``proviso term'') somewhat smaller than the others. In the initial calculation, there are precisely $4171780$ terms $-2^{a_i}$, and by combining like terms we can reduce this number to $758660$. The resulting expression is obviously too large to publish in a journal article, but it is small enough to fit comfortably in computer memory, and is therefore easily computable. A partial expansion is given in Figure~\ref{figure:day7}.

The same method works in theory to compute $|\mathcal{M}_8|$ (and higher), but the expression for $|\mathcal{M}_8|$ would have a number of terms on the order of $|\mathcal{M}_6|$. In some sense, the chained powers of two in the expression for $|\mathcal{M}_n|$ encode the entire structure of $\mathcal{M}_{n-2}$, and so we have reached the practical limit of this calculation. We now turn our attention to the abstract structure of $\mathcal{M}$ itself.

\section{The Cancellation Theorem}

\label{section:cancellation}

The next order of business is to prove that $\mathcal{M}$ is cancellative:

\begin{theorem}[Cancellation Theorem]
\label{theorem:cancellation}
If $G + T = H + T$, then $G = H$.
\end{theorem}

The Cancellation Theorem was discovered by Conway in the 1970s, and it is stated without proof in ONaG. A proof has been published once before, in Dean Allemang's 1984 thesis \cite{All84}. Since the proof is fairly tricky and is essential to the succeeding analysis, we give a full exposition here.

\begin{definition}
$H$ is a \defn{part} of $G$ if $G = H + X$ for some $X$.  In this case we say that $H + X$ is a \defn{partition} of $G$ and $X$ is the \defn{counterpart} of $H$ in~$G$.
\end{definition}

\begin{lemma}[Conway]
\label{lemma:partsofzero}
$0$ and $1$ are the only parts of $0$.
\end{lemma}

\begin{proof}
Let $X + Y = 0$ be any partition of~$0$, and assume $X$ and $Y$ to be in simplest form. Now for every option $X' \in X$, we have $X' + Y \not\linked 0$. It cannot be the case that $X'' + Y = 0$, since this would imply
\[X = X + (X'' + Y) = X'' + (X + Y) = X'',\]
contradicting the assumption that $X$ is in simplest form.  Therefore $X' + Y' = 0$ for some $Y' \in Y$, and in particular $X'$ is a part of $0$.

By induction, we may assume that $X' = 0$ or~$1$, so that the only options of $X$ are $0$ and~$1$. This implies $X = 0$, $1$, or $2$. But by an identical argument, we also have that $Y = 0$, $1$, or $2$. Since none of $2+0$, $2+1$, or $2+2$ is equal to~$0$, we conclude that $X = 0$ or~$1$, and likewise for~$Y$.
\end{proof}

\begin{definition}
$T$ is said to be \emph{cancellable} if, for all $G$ and $H$,
\begin{enumerate}
\item[(i)] $G + T = H + T$ implies $G = H$; and
\item[(ii)] $G \linked H$ implies $G + T \linked H + T$.
\end{enumerate}
\end{definition}

\begin{lemma}[Conway]
\label{lemma:cancellableparts}
If $T$ is cancellable, then so is any part of $T$.
\end{lemma}

\begin{proof}
If $T = X + Y$, then $G + X = H + X$ implies $G + T = H + T$, and $G + T \linked H + T$ implies $G + X \linked H + X$.
\end{proof}

\begin{lemma}[Conway]
\label{lemma:cancellation}
For all $T$,
\begin{enumerate}
\item[(a)] $T$ is cancellable; and
\item[(b)] $T$ has only finitely many parts.
\end{enumerate}
\end{lemma}

\begin{proof}
The proof is by induction on $T$. For $T \cong 0$, (a) is trivial and (b) follows from Lemma~\ref{lemma:partsofzero}, so assume $T \not\cong 0$. Since the assertions are independent of the form of~$T$, we can furthermore assume that $T$ is given in simplest form. We will prove (b) first, then~(a).

\vspace{0.15cm}\noindent
(b) Assume (for contradiction) that $T$ has infinitely many distinct parts $X_1,X_2,\ldots$, and write
\[T = X_1 + Y_1 = X_2 + Y_2 = \cdots.\]
Since $T \not\cong 0$, there necessarily exists an option $T' \in T$.
For each $i$ we have $T' \not\linked X_i + Y_i$, so either $T'' = X_i + Y_i$, or $T' = X_i' + Y_i$, or else $T' = X_i + Y_i'$.  The first possibility cannot occur, since it contradicts the assumption that $T$ is in simplest form.  Furthermore, $T'$ has only finitely many parts, so $T' = X_i + Y_i'$ for at most finitely many~$i$.  Thus for infinitely many values of~$i$, we have $T' = X_i' + Y_i$.  It follows that there are $m < n$ with $Y_m = Y_n$.  But $T'$ is cancellable, so by Lemma~\ref{lemma:cancellableparts} so is $Y_m$.  Since $X_m + Y_m = X_n + Y_n$, this implies $X_m = X_n$, contradicting the assumption that all $X_i$ are distinct.

\vspace{0.15cm}\noindent
(a) The proof is by induction on $G$ and~$H$.  First suppose $G \linked H$.  To show that $G + T \linked H + T$, it suffices to show that $(G+T)' \neq H+T$ and $G+T \neq (H+T)'$.  Now since $G \linked H$, we have $G' \neq H$ and $G \neq H'$, so by induction $G' + T \neq H + T$ and $G + T \neq H' + T$.  Furthermore, by induction on $T$ we have $G + T' \linked H + T'$, and this implies $G + T' \neq H + T$ and $G + T \neq H + T'$.

Next suppose $G + T = H + T$.  Then $G' + T \not\linked H + T$ and $G + T \not\linked H' + T$, so by induction $G' \not\linked H$ and $G \not\linked H'$.  To complete the proof that $G = H$, we must verify the proviso.  By symmetry, it suffices to assume $G = 0$ and show that $H = 0$ as well.

Since $G = 0$, we have $T = H + T$.  If $T = 0$, then the conclusion is immediate; otherwise, there exists an option $T' \in T$, and we have $T' \not\linked H + T$.  There are three cases.

\vspace{0.15cm}\noindent
\emph{Case 1}: $T'' = H + T$.  Then $T'' = T$, so $T'' = H + T''$, and $H = 0$ by induction.

\vspace{0.15cm}\noindent
\emph{Case 2}: $T' = H' + T$.  Then $T$ is a part of $T'$.  By induction, $T'$ is cancellable, so by Lemma~\ref{lemma:cancellableparts} so is $T$.

\vspace{0.15cm}\noindent
\emph{Case 3}: $T' = H + T^\dag$, where $T^\dag$ is an option of $T$ (possibly distinct from $T'$).  By repeated application of the identity $T = H + T$, we have
\[T = H + T = H + H + T = H + H + H + T = \cdots.\]
Since $T$ has finitely many parts, we must have $m \cdot H = n \cdot H$ for some $m < n$.  But $T'$ is cancellable and $H$ is a part of $T'$, so by Lemma~\ref{lemma:cancellableparts}, $H$~is cancellable.  Therefore $(n-m) \cdot H = 0$.  By Lemma~\ref{lemma:partsofzero}, we have $H = 0$ or $1$.  But Proposition~\ref{prop:starneq} gives $H \neq 1$.
\end{proof}

A simple corollary of the Cancellation Theorem will prove to be useful.

\begin{corollary}
\label{corollary:derivation}
Suppose $G + X = Y$, with $G$ in simplest form.  For every option $G' \in G$, either $G' + X' = Y$ or $G' + X = Y'$.
\end{corollary}

\begin{proof}
By Theorem~\ref{theorem:likeiffnotlinked}, $G' + X \not\linked Y$.  By Lemma~\ref{lemma:linkediffunlike}, either $G'' + X = Y$, or $G' + X' = Y$, or else $G' + X = Y'$.  But in the first case we have $G'' + X = G + X$, so by cancellation $G'' = G$, contradicting the assumption that $G$ is canonical.
\end{proof}

\section{Parts and Differences}

\label{section:parts}

We write $X = G - H$ to mean $G = H + X$.  By Cancellation, there is at most one such~$X$ (up to equality), so this notation is reasonable.

When we proved the Cancellation Theorem, we showed that every game has just finitely many parts.  In particular, this implies that for a fixed $G$, there are only finitely many $H$ such that $G - H$ exists.

\begin{lemma}[Difference Lemma]
\label{lemma:difference}
If $G$ and $H$ are in simplest form and $G - H$ exists, then either:
\begin{enumerate}
\item[(a)] $G - H = G' - H'$ for some $G'$ and $H'$; or else
\item[(b)] Every $G' - H$ and $G - H'$ exists, and $G - H = \{G'-H,G-H'\}$.
\end{enumerate}
\end{lemma}

\begin{proof}
Suppose $G - H$ exists, say $G = H + X$, but it is not equal to any $G' - H'$.  Now by Corollary~\ref{corollary:derivation}, for every $G'$ we have either $G' = H' + X$ or $G' = H + X'$.  But $G' = H' + X$ would imply $G' - H' = X = G - H$, which we assumed is not the case.  Therefore $G' = H + X'$, so that $G' - H$ exists and is an option of~$X$.  An identical argument shows that every $G - H'$ exists and is also an option of~$X$.

Finally, for every $X'$ we have either $G' = H + X'$ or $G = H' + X'$, so either $X' = G' - H$ or $X' = G - H'$.
\end{proof}

\begin{definition}
Let $X$ be a part of $G$.  We say that $X$ is \defn{novel} if every $G - X'$ and every $G' - X$ exists.  Otherwise we say that $X$ is \defn{derived}.  We say that a partition $G = X + Y$ is \defn{novel} if either $X$ or $Y$ is novel; \defn{derived} if both parts are derived.
\end{definition}

\begin{lemma}
\label{lemma:derivedpartition}
If $G = X + Y$ is derived, then there exist $G'$ and $X'$ such that $G' = X' + Y$.
\end{lemma}

\begin{proof}
By Corollary~\ref{corollary:derivation}, for every $G'$ we have either $G' = X' + Y$ or $G' = X + Y'$.  Likewise, for every $X'$ we have either $G' = X' + Y$ or $G = X' + Y'$.  Thus if the conclusion fails, then $X$ is novel, whence $X + Y$ is novel.
\end{proof}

The preceding analysis gives a constructive way to compute all partitions of a given~$G$ (and, therefore, a constructive calculation, or proof of nonexistence, for $G - H$ as well). In particular, for any partition $G = X + Y$, at least one of $X$ or $Y$ must be a part of some $G' \in G$; thus we can recursively compute the parts of each $G'$ to build a set of candidates for the parts of~$G$.

If a partition $X + Y$ is novel, say with $X$ novel, then necessarily $X$ is a part of every $G'$ and $Y = \{G - X',G' - X\}$. If $X + Y$ is derived, then $X =$ some $G' - Y'$ and $Y =$ some $G' - X'$, so we can find both $X$ and $Y$ among the parts of various~$G'$.

Through a more careful analysis, we can make this test fairly efficient and avoid computing unnecessary sums. A complete algorithm, due jointly to Dan Hoey and the author, is given below in Appendix~\ref{appendix:algorithm}.

\subsection*{Parity}

\begin{definition}
We say that $U$ is a \defn{unit} if it has an inverse.  If $G = H + U$ for some unit $U$, then we say that $G$ and $H$ are \defn{associates} and write $G \approx H$.
\end{definition}

By Lemma~\ref{lemma:partsofzero}, the only units are $0$ and $1$.  Thus by Proposition~\ref{prop:starneq}, every game $G$ has exactly two associates, $G$ and $G + 1$, and this induces a natural pairing among games.  We now introduce a convenient way to distinguish between the elements of each pair.

\begin{definition}
If (the simplest form of) $G$ is an option of (the simplest form of) $G + 1$, then we say that $G$ is \emph{even} and $G + 1$ is \emph{odd}.
\end{definition}

\begin{proposition}[Conway]
Every game $G$ is either even or odd, but not both.
\end{proposition}

\begin{proof}
Assume $G$ is in simplest form.  If $G$ is not even, then $G$ must be a reversible option of $G + 1$, so that $G + 1 = G'$.  Therefore $G$ is odd.

Moreover, if $G$ is even, then it is a canonical option of $G + 1$, and hence not reversible.  Therefore $G + 1 \neq G'$ for every $G'$.  So $G$ cannot be both even and odd.
\end{proof}

\begin{proposition}[Conway]~

\begin{enumerate}
\item[(a)] If $G$ and $H$ are both even or both odd, then $G + H$ is even.
\item[(b)] If $G$ is even and $H$ is odd, then $G + H$ is odd.
\end{enumerate}
\end{proposition}

\begin{proof}
Suppose $G$ and $H$ are both even, and assume (for contradiction) that $G + H$ is reversible in $G + H + 1$.  Without loss of generality, $G' + H = G + H + 1$.  By Cancellation, $G' = G + 1$, contradicting the assumption that $G$ is even.

The remaining cases follow immediately, substituting $X + 1$ for $X$ whenever $X$ is odd.
\end{proof}

\subsection*{The Group of Differences of $\mathcal{M}$}

Let $\mathcal{D}$ be the Abelian group obtained by adjoining formal inverses for all games in~$\mathcal{M}$.  That is,
\[
\mathcal{D} = \{G - H : G,H \in \mathcal{M}\},
\]
with $G_1 - H_1 = G_2 - H_2$ iff $G_1 + H_2 = G_2 + H_1$. $\mathcal{D}$ is an Abelian group, and since $\mathcal{M}$ is cancellative it embeds in~$\mathcal{D}$.

We have seen that $1 + 1 = 0$; we will now show that $1$ is the \emph{only} torsion element of $\mathcal{D}$, so that $\mathcal{D}$ is torsion-free modulo association.

\begin{theorem}
\label{theorem:torsion}
If $n \cdot G = n \cdot H$ and $G$ and $H$ are both even, then $G = H$.
\end{theorem}

\begin{proof}
Suppose $n \cdot G = n \cdot H$, written as $X$ in simplest form.  If $X = 0$, then by Lemma~\ref{lemma:partsofzero} we have $G = H = 0$.  Otherwise, there is some option $X' \in X$, and we have
\[
X' = (n-1) \cdot G + G' = (n-1) \cdot H + H' \tag{\dag}
\]
for some $G' \in G$ and $H' \in H$. Multiplying by $n$ gives
\[
n \cdot (n-1) \cdot G + n \cdot G' = n \cdot (n-1) \cdot H + n \cdot H'.
\]
Now $n \cdot G = n \cdot H$, so $n \cdot (n-1) \cdot G = n \cdot (n-1) \cdot H$, so by Cancellation
\[
n \cdot G' = n \cdot H'.
\]
Now $G'$ and $H'$ have the same parity, by (\dag).  So by induction on $G$ and $H$, we may assume that $G' = H'$.  But now Cancellation on (\dag) gives
\[
(n-1) \cdot G = (n-1) \cdot H,
\]
and the conclusion follows by induction on $n$.
\end{proof}

\begin{corollary}
$1$ is the only torsion element of $\mathcal{D}$.
\end{corollary}

\begin{proof}
Let $G - H \in \mathcal{D}$ and suppose $n \cdot (G-H) = 0$ for some~$n$, so that $n \cdot G = n \cdot H$. Theorem~\ref{theorem:torsion} implies $G \approx H$, so that $G - H = 0$ or~$1$.
\end{proof}

In particular $\mathcal{D}/\approx$ is a torsion-free Abelian group, but the following major question remains open.

\begin{question}
What is the isomoprhism type of $\mathcal{D}$?
\end{question}

\section{Primes}

\label{section:primes}

\begin{definition}
A part $H$ of $G$ is said to be \defn{proper} if $H \not\approx 0$ or~$G$.
\end{definition}

\begin{definition}
$G$ is said to be \defn{prime} if $G$ is not a unit and $G$ has no proper parts.
\end{definition}

\begin{theorem}
Every game $G$ can be partitioned into primes.
\end{theorem}

\begin{proof}
By Lemma~\ref{lemma:cancellation}, every game has just finitely many parts.  We can therefore prove the theorem by induction on the number of proper parts of $G$.

If $G$ itself is prime, then there is nothing to prove.  Otherwise, we can write $G = X + Y$, where $X$ and $Y$ are proper parts of $G$.  Now every proper part of~$X$ is a proper part of $G$, but $X$ is not a proper part of $X$.  Therefore $X$ has strictly fewer proper parts than $G$.  By induction, $X$ has a prime partition.  By the same argument, so does $Y$, and we are done.
\end{proof}

It is important to note that a partition of $G$ into primes need not be unique.  For example, one can show that
\[(4+2)_\sh = 2 + P = 4 + Q,\]
where $P$ and $Q$ are distinct primes. (This example is originally due to Conway and Norton.) The behavior of primes can often be quite subtle.  It is possible for $G$ to have several prime partitions of different lengths:
\[(4+2_\sh)_\sh = 2 + P_1 + P_2 = 4 + Q,\]
where $P_1$, $P_2$, and $Q$ are all distinct primes.  Furthermore, $G + G$ might have a prime part that is not a part of $G$.  For example, if $G = (4+2_\sh)_{\sh\sh}$, then there exists a partition of $G+G$ into exactly three primes.

While these examples advise caution, we can nonetheless discern some useful structure among primes.  In the following propositions we assume $G$ to be given in simplest form.

\begin{proposition}
If $G$ has $0$ or $1$ as an option, then $G$ is prime.
\end{proposition}

\begin{proof}
Let $G' \in G$ with $G' = 0$ or $1$, and suppose that $G = X + Y$.  By Corollary~\ref{corollary:derivation}, we have $G' = X' + Y$, without loss of generality.  But $G' \approx 0$, so by Lemma~\ref{lemma:partsofzero}, $Y \approx 0$, and hence $X \approx G$.  Thus $0$ and $G$ are the only even parts of $G$.
\end{proof}

\begin{proposition}
\label{prop:primeoptbiprime}
If $G$ has a prime option, then $G$ has at most two even prime parts.
\end{proposition}

\begin{proof}
Fix a prime option $G'$. First suppose $G = X + Y + Z$ for any three games $X$, $Y$, and~$Z$. By Corollary~\ref{corollary:derivation}, we have $G' = X' + Y + Z$, without loss of generality. Since $G'$ is prime, one of $Y$ or $Z$ must be a unit. This shows that every partition of $G$ involves at most two primes.

Now suppose we write $G = P_1 + P_2 = Q_1 + Q_2$.  Without loss of generality,
\[
G' = P_1' + P_2 = Q_1' + Q_2.
\]
Since $G'$ is prime, $P_1'$ and $Q_1'$ must be units, so $P_2 \approx Q_2$ are equal up to a unit. By Cancellation, $P_1 \approx P_2$ as well.
\end{proof}

\begin{corollary}
If $G$ has at least three prime options, distinct modulo associatation, then $G$ is prime.
\end{corollary}

\begin{proof}
Suppose $G$ has a prime option but is not itself prime.  By Proposition~\ref{prop:primeoptbiprime}, $G$ has a unique prime partition $G = P + Q$.  Therefore every $G' = P' + Q$ or $P + Q'$.  In particular, if $G'$ is prime, then $G' \approx P$ or~$Q$.  So $G$ has at most two prime options, up to association.
\end{proof}

In a related vein, we have the following proposition.

\begin{proposition}[Conway]
If every option of $G$ is prime, then so is $G$, \emph{unless} $G = 0$, $2_\sh$, $3_\sh$, or $32$.
\end{proposition}

\begin{proof}
Suppose every option of $G$ is prime, but $G$ is not.  By Lemma~\ref{prop:primeoptbiprime}, if $G \neq 0$ then we can write $G = P + Q$ for suitable primes $P$ and $Q$, and furthermore $P$ and $Q$ are unique (up to association). Assume each of $G$, $P$, and $Q$ is given in simplest form.

Now $G$ cannot be odd, since then $G + 1$ would be a prime option of~$G$.  So $G$ is even, and we may therefore assume that $P$ and $Q$ are both even.  Now for every option $P'$, we have either $G' = P' + Q$ or $G = P' + Q'$.  If $G' = P' + Q$, then since $G'$ is prime we must have $G' \approx Q$, so $P'$ is a unit.  Suppose instead that $G = P' + Q'$.  We cannot have $P' \approx P$, since $P$ is even.  Furthermore, we cannot have $P' \approx Q$: since the partition of $G$ into $P+Q$ is unique, this would imply $Q' \approx P$, so that $P'' \approx P$, contradicting the assumption that $P$ is in simplest form.  Therefore either $P'$ is a unit and $Q' \approx G$, or else $Q'$ is a unit and $P' \approx G$.

We have therefore shown that every option of $P$ is either $0$, $1$, $G$, or $G + 1$.  By symmetry, the same is true for $Q$.  But for every option $G'$, we have $G' = P' + Q$ or $G' = P + Q'$.  Since $G'$ is prime, this implies $G' \approx P$ or $Q$.  Therefore $G$ cannot be an option of \emph{both} $P$ and $Q$: this would imply that $P$ (or $Q$) is associated to one of its followers.  Therefore one of $P$ or $Q$ has only $0$ and $1$ as options.  Without loss of generality, assume it is~$P$.  Since $P$ is prime, we must have $P = 2$.

If also $Q = 2$, then $G = 2 + 2 = 32$ and we are done.  Otherwise, $Q$ has $G$ as a follower.  But this means no follower of $G$ can be associated to $Q$.  Thus every option of $G$ is associated to $P = 2$, and this completes the proof.
\end{proof}

The above propositions suggest that composite games are relatively rare.  This can be made precise by considering the number of composite games born by day~$n$.  It suffices to consider only even composites, since the number of odd composites born by day~$n$ is precisely equal to the number of even composites born by day~$n-1$.

There are six even composites born by day~$4$.  These and their unique partitions are summarized in Figure~\ref{figure:day4composites}.  A computer search revealed exactly $490$ even composites born by day~$5$.  Of these, $481$ have exactly two even prime parts.  Figure~\ref{figure:day5composites} lists the nine examples with more than two parts.

\begin{figure}
\[\begin{array}{c@{~=~}l}
2_2 & 2 + 2 \\
2_\sh & 2 + 2_\sh2_{\sh1}2_2 \\
3_\sh & 2 + 3_\sh3_{\sh1}3_2
\end{array}
\qquad
\begin{array}{c@{~=~}l}
2_\sh2 & 2 + 2_2 2_{\sh2} (2_\sh2)(2_\sh2)_1 \\
2_\sh32 & 2 + 2_2 3_2 2_{\sh 2} (2_\sh32) (2_\sh32)_1 \\
2_{\sh\sh} & 2 + 2_{\sh\Neg2} + (2_{\sh\sh}2_{\sh\sh1}2_{\sh\Neg2})0
\end{array}\]
\vspace{-0.25cm}
\caption{The six composite even games born by day 4.}
\label{figure:day4composites}
\end{figure}

\begin{figure}
\[\begin{array}{c|c}
G & \textrm{Primes} \\
\hline
2_{\sh\sh}  & 3 \\
3_{\sh\sh}  & 3 \\
2_{\sh1\sh} & 3
\end{array}
\qquad
\begin{array}{c|c}
G & \textrm{Primes} \\
\hline
2_{2\sh}      & 3 \\
2_{\sh 2}     & 3 \\
2_{\sh\sh\sh} & 4
\end{array}
\qquad
\begin{array}{c|c}
G & \textrm{Primes} \\
\hline
2_{\sh 1}2_\sh & 3 \\
2_{\sh\sh}2_\sh & 3 \\
2_{\sh\sh}2_{\sh 1}2_\sh & 3
\end{array}\]
\vspace{-0.25cm}
\caption{The nine highly composite even games born by day 5, listed with number of prime parts.}
\label{figure:day5composites}
\end{figure}

\subsection*{Unique Partitions}

\begin{definition}
We say that $G$ has the \emph{unique partition property (UPP)} if~$G$ has exactly one prime partition (up to association).
\end{definition}

We have already noted that $(4+2)_\sh$ does not have the UPP.  In this section we will prove that every game born by day~6 has the UPP.  Since $(4+2)_\sh$ is born on day~7, it is therefore a minimal example.

\begin{definition}
We say $G$ is a \emph{biprime} if $G$ has exactly two even prime parts.
\end{definition}

\begin{proposition}
\label{proposition:nubiprime}
Suppose that $G$ has a biprime option, say $G' = R + S$, with $R$ and $S$ prime.  Then either:
\begin{enumerate}
\item[(a)] $G$ is itself a prime or a biprime; or
\item[(b)] There is a prime $P$ such that $G = P + R + S$, and this is the unique prime partition of~$G$; or
\item[(c)] There are primes $P$ and $Q$ such that $G = P + R = Q + S$, and these are the only two prime partitions of~$G$.
\end{enumerate}
\end{proposition}

\begin{proof}
If $G$ is prime, then we are in case (a), so assume that it is not.

\vspace{0.1in}\noindent
\emph{Case 1}:
First suppose $G$ has a prime partition
\[G = P_1 + \cdots + P_k\]
with $k \geq 3$.  Without loss of generality,
\[G' = P_1' + P_2 + \cdots + P_k.\]
Since $G'$ is a biprime, it must be the case that $k = 3$ and $P_1'$ is a unit, and without loss of generality $P_2 \approx R$ and $P_3 \approx S$.  We claim that this is the unique prime partition of $G$.  For suppose
\[G = Q_1 + \cdots + Q_l\]
is any prime partition.  By an identical argument we have $l \leq 3$ and
\[G' = Q_1' + Q_2 + \cdots + Q_l.\]
Certainly $l \neq 1$, since $G$ is composite. Moreover, since $Q_2$ is a prime part of $G'$, we have $Q_2 \approx R$ without loss of generality.  Therefore $l \neq 2$: by Cancellation, $l = 2$ would imply $Q_1 = P_1 + R$, contradicting the fact that $Q_1$ is prime.  So $l = 3$, and hence $Q_2 \approx R$, $Q_3 \approx S$, and by Cancellation $Q_1 \approx P_1$. This shows that $P_1 + P_2 + P_3$ is unique, establishing case~(b).

\vspace{0.1in}\noindent
\emph{Case 2}:
Next suppose that every prime partition of $G$ has exactly two primes. Consider any such partition
\[
G = P_1 + P_2.
\]
Without loss of generality,
\[
G' = P_1' + P_2.
\]
But by the assumptions on $G'$, this implies $P_2 = R$ or~$S$, and by Cancellation, $P_1$ is uniquely determined by~$P_2$. This shows that there are at most two such partitions, so either $G$ is a biprime, or else it has exactly two prime partitions into exactly two parts, as in~(c).
\end{proof}

\begin{corollary}
\label{corollary:nusimple}
Suppose $G$ is a game born by day $6$ without the UPP.  Assume that at least one option of $G$ is a biprime.  Then in fact,
\[G = 2 + P = Q_1 + Q_2,\]
for distinct primes $P$, $Q_1$, and $Q_2$, none of which equal~$2$.
\end{corollary}

\begin{proof}
This is just Proposition~\ref{proposition:nubiprime}, together with the (computationally verifiable) fact that $2$ is a part of every composite game born by day~5.
\end{proof}

\begin{lemma}
\label{lemma:nudecomposition}
Suppose $G$ is a game born by day $6$ without the UPP.  Suppose $G$ has a biprime option $2+R$.  If some other option $G'$ does \emph{not} have $R$ as a part, then $G = R+S$, for some part $S$ of $G'$.
\end{lemma}

\begin{proof}
By Corollary~\ref{corollary:nusimple}, we have
\[G = 2 + P = Q_1 + Q_2,\]
where $Q_1,Q_2 \neq 2$.  Without loss of generality, $2 + R = Q_1 + Q_2'$.  Since $Q_1 \neq 2$, we must have $Q_1 = R$.

Now we cannot have $G' = Q_1 + Q_2'$, since $Q_1 = R$ is not a part of $G'$.  So $G' = Q_1' + Q_2$, whence $Q_2$ is a part of $G'$.
\end{proof}

\begin{theorem}
Every game born by day 6 has the UPP.
\end{theorem}

\begin{proof}
Let $G$ be a game born by day 6.  If $G$ has any prime options, then $G$ is a biprime, so it has the UPP.  Likewise, if $G$ is odd, then $G + 1$ is an even game born by day~5, with the same parts as $G$.  We know that every game born by day~5 has the UPP, so such $G$ must also have the UPP.  Thus we need only consider even games whose options are all composite.

Now let
\[\mathcal{C} = \{G \in \mathcal{M}_5 : \textrm{$G$ has at least $3$ prime parts}\}.\]
We noted previously that $|\mathcal{C}| = 10$.  Thus there are $2^{10}$ subsets of $\mathcal{C}$, and a computer search can rapidly verify that all of them have the UPP.

This leaves only those games with at least one biprime option.  We can now apply the following trick.  Let
\[\mathcal{A} = \{H : H \textrm{ is an even prime part of some composite game born by day } 5\}.\]
Let $\mathcal{A} + \mathcal{A}$ be the set of all pairwise sums of elements of $\mathcal{A}$.  If $G$ has at least one biprime option, then by Lemma~\ref{lemma:nudecomposition}, \emph{either}:
\begin{enumerate}
\item[(i)] $G \in \mathcal{A} + \mathcal{A}$; \emph{or}
\item[(ii)] All options of $G$ share a common part $R \neq 2$.
\end{enumerate}

It therefore suffices to exhaust all possibilities for (i) and (ii).  For (i), we have $|\mathcal{A}| < 500$, so $|\mathcal{A} + \mathcal{A}| < 25000$.  It is therefore easy to compute the set $\mathcal{A} + \mathcal{A}$, and a simple computation shows that for most $G \in \mathcal{A} + \mathcal{A}$, we have $\birthday(G) > 6$.  We can then show directly that the remaining few have the UPP.

To complete the proof, we describe how to exhaust case~(ii).  For each $R \in \mathcal{A}$, let
\[\mathcal{C}_R = \{G \in \mathcal{C} : R \textrm{ is a part of } G\}.\]
Now $\Sigma_R |\mathcal{C}_R|$ is small, since the elements of $\mathcal{C}$ collectively have a small number of parts.  But to address case~(ii), we need only consider those games whose options are subsets of
\[\{2 + R\} \cup \mathcal{C}_R,\]
for some $R \neq 2$: these are exactly the games whose options share the common factor $R$.  We can therefore iterate over all $R$ and all subsets of $\{2 + R\} \cup \mathcal{C}_R$, checking that each possibility has the UPP.

All of the necessary computations to complete the proof have been implemented and verified in \emph{cgsuite}.
\end{proof}

\appendix

\section{Recurrence Relations for $|\mathcal{M}_n|$}

\label{appendix:daynrecurrences}

This Appendix gives a set of recurrence relations that can be used to compute the exact value of $|\mathcal{M}_n|$, expressed in terms of chained powers of~$2$. (As discussed in Section~\ref{section:dayn}, the calculation can feasibly be carried out only for $n \leq 7$, but the relations continue to hold for larger~$n$.)

For $n \geq 0$ and $G,K \in \mathcal{M}$, define:
\[
\begin{array}{c@{\quad=\quad}l}
\mathcal{R}_n^G & \left\{H \subset \mathcal{M}_{n-1} : H \text{ simplifies to } G\right\} \bigstrut \\
\mathcal{S}_n^K & \left\{H \in \mathcal{M}_n : K \in H\right\} \bigstrut \\
\mathcal{R}_n^{G,K} & \left\{H \subset \mathcal{M}_{n-1} : H \text{ simplifies to } G \text{ and } K \in H\right\} \bigstrut \\
\mathcal{N}_n & \left\{H \in \mathcal{M}_n : H \text{ is an $\mathscr{N}$-position}\right\} \bigstrut \\
\mathcal{N}_n^0 & \left\{H \in \mathcal{M}_n : H \text{ is an $\mathscr{N}$-position and } 0 \in H\right\} \bigstrut
\end{array}
\]
Here $\mathcal{R}_n^{G,K}$ is defined only when $G \in K$.

\begin{theorem}
For all $n \geq 2$, we have the following recurrences.
\[
\begin{array}{c@{\quad=\quad}>{\displaystyle}l}
|\mathcal{M}_n| & 2^{|\mathcal{M}_{n-1}|} - \sum_{G \in \mathcal{M}_{n-2}} |\mathcal{R}_n^G| \vspace{0.2in} \\
|\mathcal{R}_n^G| &
\begin{cases}
2^{|\mathcal{S}_{n-1}^G|} - 1 & \text {if } G \not\cong 0 \\
2^{|\mathcal{S}_{n-1}^G|} - 2^{|\mathcal{N}_{n-1}^0|} & \text {if } G \cong 0 \\
\end{cases} \vspace{0.2in} \\
|\mathcal{S}_n^K| &
2^{|\mathcal{M}_{n-1}| - 1} - \sum_{G \in K}|\mathcal{R}_n^{G,K}| - \sum_{G \in \mathcal{S}_{n-2}^K}|\mathcal{R}_n^G| \vspace{0.2in} \\
|\mathcal{R}_n^{G,K}| &
\begin{cases}
2^{|\mathcal{S}_{n-1}^G|-1} & \text{if } G\not\cong 0 \text{ or } K \text{ is a \P-position} \\
2^{|\mathcal{S}_{n-1}^G|-1} - 2^{|\mathcal{N}_{n-1}^0|-1} & \text{if } G \cong 0 \text{ and } K\text{ is an \N-position}
\end{cases} \vspace{0.1in} \\
\multicolumn{2}{l}{\qquad\qquad\qquad\textup{($|\mathcal{R}_n^{G,K}|$ is defined only when $G \in K$)}} \vspace{0.2in} \\
|\mathcal{N}_n| &
2^{|\mathcal{M}_{n-1}|} - 2^{|\mathcal{N}_{n-1}|} + 1 - \sum_{G \in \mathcal{N}_{n-2}}|\mathcal{R}_n^G| \vspace{0.2in} \\
|\mathcal{N}_n^0| &
2^{|\mathcal{M}_{n-1}|-1} - 2^{|\mathcal{N}_{n-1}|-1} - \sum_{G \in \mathcal{N}_{n-2}^0}|\mathcal{R}_n^G|
\end{array}
\]
\end{theorem}

\begin{proof}
We discuss each equation in turn.

\vspace{0.1in}\noindent
$|\mathcal{M}_n|$. There are $2^{|\mathcal{M}_{n-1}|}$ subsets of $\mathcal{M}_{n-1}$. For each subset $H \subset \mathcal{M}_{n-1}$, either $H$ is canonical, or else $H$ simplifies to $G$ for a unique $G \in \mathcal{M}_{n-2}$.

\vspace{0.1in}\noindent
$|\mathcal{R}_n^G|$. In order for $H \subset \mathcal{M}_{n-1}$ to simplify to~$G$, it must have the exact form
\[
H \cong \{G_1,\ldots,G_k,H_1,\ldots,H_l\},
\]
where $G \cong \{G_1,\ldots,G_k\}$ and $G \in$ each $H_j$. Now $G_1,\ldots,G_k$ are fixed, so games with this form correspond one-to-one with subsets $\{H_1,\ldots,H_l\}$ of $\mathcal{S}_{n-1}^G$. There are in total $2^{|\mathcal{S}_{n-1}^G|}$ such subsets, but we must subtract $1$ to exclude the empty set, which corresponds to $G$ itself.

If $G \not\cong 0$, then we are done. If $G \cong 0$, then we must apply the proviso as well: in order for $H$ to be reducible, it must also satisfy $o(H) = \N$, in addition to having the proscribed form. So we must subtract off a count of $H$ with $o(H) = \P$. Now $o(H) = \P$ if and only if every $H_j$ is an \N-position and there is at least one~$H_j$. There are $|\mathcal{N}_{n-1}^0|$ such \N-positions (since we also know that $0 \in$ each $H_j$), yielding a total of $2^{|\mathcal{N}_{n-1}^0|} - 1$ possibilities for~$H$. Subtracting this from the overall count of $2^{|\mathcal{S}_{n-1}^G|} - 1$ yields the total stated in the theorem.

\vspace{0.1in}\noindent
$|\mathcal{S}_n^K|$. There are $2^{|\mathcal{M}_{n-1}|-1}$ subsets $H \subset \mathcal{M}_{n-1}$ with $K \in H$. There are two ways that such an $H$ might be reducible: either $H$ simplifies to some $G \in K$ (so that $K$ reverses through~$G$), or else $H$ simplifies to some other $G \in \mathcal{M}_{n-2}$ with $K \in G$ (so that $K$ is already present in the simplified form~$G$). These are mutually exclusive, since in the first case $K \not\in G$, but in the second $K \in G$.

\vspace{0.1in}\noindent
$|\mathcal{R}_n^{G,K}|$. There are $2^{|\mathcal{M}_{n-1}|-1}$ subsets $H \subset \mathcal{M}_{n-1}$ with $K \in H$. To be reducible, $H$ must have the exact form
\[
H \cong \{G_1,\ldots,G_k,H_1,\ldots,H_j,K\}
\]
with $G \in$ each $H_j$. (The assumption $G \in K$ implies that $K$ is not among the options $G_1,\ldots,G_k$ of~$G$.) The rest of the argument proceeds just as for $|\mathcal{R}_n^G|$, with two modifications: (i)~$j = 0$ is no longer a special case; since $K \in H$, it will never be true that $H \cong G$. (ii)~If $G \cong 0$ but $o(K) = \P$, then the proviso does not apply, since the presence of $K$ ensures $o(H) = \N$.

\vspace{0.1in}\noindent
$|\mathcal{N}_n|$. There are $2^{|\mathcal{M}_{n-1}|}$ subsets $H \subset \mathcal{M}_{n-1}$. There are just two ways that $H$ might fail to be a canonical \N-position. Either $H$ is a \emph{nonempty} set of \N-positions, in which case it's a \P-position; or else $H$ simplifies to some other \N-position $G \in \mathcal{N}_{n-2}$. These are mutually exclusive, since no \P-position can simplify to an \N-position.

\vspace{0.1in}\noindent
$|\mathcal{N}_n^0|$. This one is similar to the preceding argument, with one small twist. There are $2^{|\mathcal{M}_{n-1}|-1}$ subsets $H \subset \mathcal{M}_{n-1}$ with $0 \in H$. There are two ways that $H$ might fail to be a canonical \N-position. Either $H$ is a set of \N-positions with $0 \in H$ (which must always be nonempty); or else $H$ simplifies to some other \N-position $G \in \mathcal{N}_{n-2}$. But if $H$ simplifies to some other~$G$, then it must be the case that $0 \in G$ as well (since $0$ can never be reversible). So these two conditions are both mutually exclusive and entirely contained within the $2^{|\mathcal{M}_{n-1}|-1}$ subsets originally counted.
\end{proof}

\section{An Algorithm for Computing Parts}

\label{appendix:algorithm}

Algorithm~\ref{algorithm:parts}, due jointly to Dan Hoey and the author, shows how to compute the parts of an arbitrary game $G$ efficiently. The algorithm is structured so that whenever a part $X$ is detected, then so is its counterpart $G - X$. A record of these part-counterpart relationships can be kept throughout the execution of the algorithm, so that differences of the form $G - X$ can be efficiently resolved.

\newcommand{\Parts}{\mathbf{Parts}}
\begin{algorithm}[tbp]
\hrule\vspace{0.15cm}
\begin{algorithmic}[1]
\State $\mathcal{X} \gets \emptyset$
\ForAll{$G' \in G$}
  \State Recursively compute $\Parts(G')$
\EndFor
\ForAll{$X$ such that $X \in \Parts(G')$ for some $G' \in G$} \label{step:mainloop}
  \If{$X$ is in every $\Parts(G')$ and every $X' \in \mathcal{X}$}
    \State $Y \gets \{G - X',G' - X\}$
    \State $\mathcal{X} \gets \mathcal{X} \cup \{X,Y\}$ \Comment{$X+Y$ is novel}
           \label{step:expandXnovel}
  \Else
    \If{$X \not\in \Parts(G')$ for some $G'$}
      \State Fix any such $G'$
      \State $\mathcal{Y} \gets \{G' - X' : X' \in \Parts(G')\}$
    \Else
      \State Fix any $X' \not\in \mathcal{X}$
      \State $\mathcal{Y} \gets \{G' - X' : X' \in \Parts(G')\}$
    \EndIf
    \ForAll{$Y \in \mathcal{Y}$}
      \vspace{0.1cm}
      \If{$\left\{\begin{tabular}{@{}l@{~}l@{}}
                       & $\forall G' \in G$, either $G' - Y = X'$ or $G' - X = Y'$; \\
          \textbf{and} & $\forall X' \in X$, either $G' - Y = X'$ or $(X',Y') \in \mathcal{X}$; \\
          \textbf{and} & $\forall Y' \in Y$, either $G' - X = Y'$ or $(X',Y') \in \mathcal{X}$
          \end{tabular}\right\}$}
          \label{step:bigif}
      \vspace{0.1cm}
        \State $\mathcal{X} \gets \mathcal{X} \cup \{X,Y\}$ \Comment{$X+Y$ is derived}
               \label{step:expandXderived}
      \EndIf
    \EndFor
  \EndIf
\EndFor
\If{$\mathcal{X}$ has changed}
  \State Return to Step~\ref{step:mainloop}
\EndIf
\State $\Parts(G) \gets \mathcal{X}$
\end{algorithmic}
\vspace{0.15cm}\hrule
\caption{Computing the parts of $G$.}
\label{algorithm:parts}
\end{algorithm}

\begin{theorem}
Algorithm~\ref{algorithm:parts} correctly computes $\Parts(G)$.
\end{theorem}

\begin{proof}
We can assume that $\Parts(G')$ is correctly computed for all $G' \in G$.  Now it is easy to see that every game that is put into $\mathcal{X}$ is indeed a part of~$G$: if a partition $X+Y$ is added in Step~\ref{step:expandXnovel}, then $X+Y$ is novel; if it's added in Step~\ref{step:expandXderived}, then the condition of Step~\ref{step:bigif} directly witnesses the identity $G = X + Y$.

To complete the proof, we must show that every part of $G$ is eventually placed in $\mathcal{X}$. Suppose not, and let $X+Y$ be a partition of $G$ that the algorithm fails to find. Assume that $X+Y$ is minimal in the sense of $\birthday(X)+\birthday(Y)$. In particular, at some stage of the algorithm we have $X',Y' \in \mathcal{X}$ for every partition of the form $G = X' + Y'$.

Suppose $X+Y$ is derived.  By Lemma~\ref{lemma:derivedpartition}, there is some $G^\dag$ such that $G^\dag = X + Y'$.  Therefore $X \in \Parts(G^\dag)$, so $X$ will be encountered in the main loop of Algorithm~\ref{algorithm:parts}.  Since $X$ is derived, either some $G - X'$ or some $G' - X$ must fail to exist.  If $G - X'$ does not exist, then since $G = X + Y$, we must have $G' = X' + Y$ for some $G'$.  Therefore $Y = G' - X'$.  If $G' - X$ does not exist, then $G' = X' + Y$ for some $X'$, so again $Y = G' - X'$.  In either case $Y \in \mathcal{Y}$ (as defined in Algorithm~\ref{algorithm:parts}).  Thus it suffices to verify that $G$, $X$, and $Y$ jointly satisfy the condition of Step~\ref{step:bigif}.  But by the inductive hypothesis, we have $X',Y' \in \mathcal{X}$ whenever $G = X' + Y'$, so the condition asserts precisely that $G = X + Y$, which is true.  Therefore $X$ and $Y$ are put into $\mathcal{X}$ in Step~\ref{step:expandXderived}, a contradiction.

Finally, suppose $X+Y$ is novel, and assume without loss of generality that~$X$ is novel.  Then every $G - X'$ and $G' - X$ exists, and it is easily checked that $Y = \{G-X',G'-X\}$.  Since the algorithm fails to detect $X+Y$ at Step~\ref{step:expandXnovel}, it must be the case that $X' \not\in \mathcal{X}$ for some $X'$.  Now $\birthday(X') < \birthday(X)$, so by the inductive hypothesis, this implies $\birthday(G-X') > \birthday(Y)$.  Therefore $G-X'$ must be a reversible option of $Y = \{G-X',G'-X\}$, so either $G-X'' = Y$ or $G'-X' = Y$.  The former is obviously false (by Cancellation), so we must have $Y = G'-X'$.  But then the partition $X+Y$ will be detected in Step~\ref{step:expandXderived}, by the same argument used in the previous paragraph.
\end{proof}

\bibliography{games}

\begin{thebibliography}{10}

\bibitem{All84}
D.~T. Allemang.
\newblock Machine computation with finite games.
\newblock Master's thesis, Trinity College, Cambridge, 1984.
\newblock \url{http://miseregames.org/allemang/}.

\bibitem{Bou01}
C.~L. Bouton.
\newblock Nim, a game with a complete mathematical theory.
\newblock {\em Ann. of Math.}, 3(2):35--39, 1901.

\bibitem{Con01}
J.~H. Conway.
\newblock {\em On Numbers and Games}.
\newblock A~K Peters, Ltd.~/ CRC Press, Natick, MA, second edition, 2001.

\bibitem{Gru39}
P.~M. Grundy.
\newblock Mathematics and games.
\newblock {\em Eureka}, 2:6--8, 1939.

\bibitem{GrS56}
P.~M. Grundy and C.~A.~B. Smith.
\newblock Disjunctive games with the last player losing.
\newblock {\em Proc. Cambridge Philos. Soc.}, 52:527--533, 1956.

\bibitem{GS56}
R.~K. Guy and C.~A.~B. Smith.
\newblock The {$G$-values} of various games.
\newblock {\em Proc. Cambridge Philos. Soc.}, 52:514--526, 1956.

\bibitem{Pla05}
T.~E. Plambeck.
\newblock Taming the wild in impartial combinatorial games.
\newblock {\em INTEGERS: The Electr. J. Combin. Number Theory}, 5(1):\#G05,
  2005.

\bibitem{PS08}
T.~E. Plambeck and A.~N. Siegel.
\newblock Mis\`ere quotients for impartial games.
\newblock {\em J. Combin. Theory Ser.~A}, 115(4):593--622, May 2008.

\bibitem{Sie13}
A.~N. Siegel.
\newblock {\em Combinatorial Game Theory}.
\newblock Number 146 in Graduate Studies in Mathematics. American Mathematical
  Society, 2013.

\bibitem{Smi}
C.~A.~B. Smith.
\newblock Compound two-person deterministic games.
\newblock unpublished manuscript.

\bibitem{Tho99}
C.~Thompson.
\newblock Count of day 6 misere-inequivalent impartial games, 1999.
\newblock posted to usenet \texttt{rec.games.abstract} on February 19, 1999.

\end{thebibliography}

\end{document}